\documentclass[12pt]{amsart}
\usepackage{amsmath,amssymb,amsthm}
\usepackage{amscd}

\usepackage{comment}
\usepackage{graphicx}
\usepackage{color}
\usepackage[abs]{overpic} 
\usepackage{ulem} %under line
\usepackage[dvipdfm,left=25truemm,right=25truemm,top=30truemm,bottom=30truemm]{geometry}

\usepackage{bm}
\theoremstyle{plain} 
\newtheorem{theorem}{Theorem}[section] 
\newtheorem{lemma}[theorem]{Lemma}

\newtheorem{proposition}[theorem]{Proposition}

\theoremstyle{definition}
\newtheorem{definition}[theorem]{Definition}
\newtheorem{remark}[theorem]{Remark}

\makeatletter
\def\Bline{%
\noalign{\ifnum0=`}\fi\hrule \@height 1pt \futurelet
\reserved@a\@xhline}
\@addtoreset{equation}{section}

\makeatother

\allowdisplaybreaks[3]

\newcommand{\textR}[1]{#1}

\newcommand{\R}{\mathbb{R}}

\newcommand{\pd}[2]{\dfrac{\partial#1}{\partial#2}}

\renewcommand{\phi}{\varphi}
\renewcommand{\epsilon}{\varepsilon}

\newcommand{\wtilde}[1]{\widetilde{#1}}

\renewcommand{\geq}{\geqslant}
\renewcommand{\leq}{\leqslant}

\newcommand{\ma}{\mathcal{A}}
\newcommand{\wt}[1]{\widetilde{#1}}
\newcommand{\tk}[1]{\widetilde{\kappa}_{#1}}
\newcommand{\tv}[1]{\widetilde{\bm{v}}_{#1}}

\begin{document}

\title[Height functions on singular surfaces]
{Height functions on singular surfaces parameterized by smooth maps $\mathcal{A}$-equivalent to $S_k$, $B_k$, $C_k$ and $F_4$.}

\author[T.~Fukui and M.~Hasegawa]
{Toshizumi Fukui and Masaru Hasegawa}

\address[Toshizumi Fukui]{%
Departmet of Mathematics, Faculty of Science, Saitama University, Saitama, 338-8570, Japan.
}
\email{tfukui@rimath.saitama-u.ac.jp}

\address[Masaru Hasegawa]{%
%   Instituto de Ci\^encias Matem\'aticas e 
%   de Computa\c{c}\~ao - USP,
%   Avenida Trabalhador S\~ao-carlense, 400- Centro,
%   CEP:13566-590 - S\~ao Carlos - SP, Brazil.
Department of Information Science, Center for Liberal Arts and Sciences,
Iwate Medical University, 2-1-1, Nishi-Tokuda, Yahaba-cho, Shiwa-gun, 028-3694, Iwate, Japan
}
\email{mhase@iwate-med.ac.jp}



\thanks{
The first author is partially supported by grant-in-Aid in Science 26287011.
The second author was supported by the FAPESP post-doctoral grant number 2013/02543-1 during a post-doctoral period at ICMC-USP} 

\begin{abstract}
%We investigate singularities of dual surfaces of singular surfaces parameterized by smooth map-germs $\mathcal{A}$-equivalent to one of $S_k$, $B_k$, $C_k$ and $F_4$ singularities in terms of extended geometric language via finite succession of blowing-ups.
%We also investigate singularities of height functions on our surfaces.
We describe singularities of height functions on singular surfaces in $\R^3$ parameterized by smooth map-germs $\mathcal{A}$-equivalent to one of $S_k$, $B_k$, $C_k$ and $F_4$ singularities in terms of extended geometric language via finite succession of blowing-ups. 
We investigate singularities of dual surfaces of such singular surfaces.
\end{abstract}

\maketitle

\section{Introduction.}
 
Singular surfaces are studied as objects in the extrinsic differential geometry (for example, \cite{BW1998, FH2012, FH2013, FH2015, MN-B2014, NT2007, O-ST2012, Tari2007, Teramoto2015, West1995}). 
The distance-squared functions and height functions are fundamental tools in such researches. 
In our previous work \cite{FH2015}, we investigate the family of distance-squared functions on singular surfaces with $S_k$, $B_k$, $C_k$ and $F_4$ singularities. 
In this paper, we are going to investigate height functions on such singular surfaces. 

Let $f:(\R^2,0)\to(\R^3,0)$ be a smooth map-germ which parameterize a surface $S$ (possibly with singularities) in $\R^3$.
We consider families $H:(\R^2\times S^2,(0,\bm{v}_0))\to\R$ defined by %: 
\[
H(u,v,\bm{v})=\langle f(u,v),\bm{v} \rangle,
\]
and $\wt{H}:(\R^2\times S^2 \times \R, (0,\bm{v}_0,t_0))\to\R$ defined by
$$
\wt{H}(u,v,\bm{v},t) = H(u,v,\bm{v})-t = \langle f(u,v),\bm{v}\rangle - t,
$$
where $S^2$ is the unit sphere in $\R^3$ and $\langle\ ,\ \rangle$ denotes the Euclidean inner product in $\R^3$.
We define $h_{\bm{v}}(u,v)= H(u,v,\bm{v})$, which is the height function on $S$ along $\bm{v}$, and also $h_{\bm{v},t}(u,v) = H(u,v,\bm{v},t)$, which is the extended height function on $S$ along $\bm{v}$. 
These families are important, since the bifurcation set of $H$ is the singular values of the Gauss map of $S$ and the discriminant set of $\wt{H}$ is isomorphic to the dual surface of $S$. 
For regular surfaces (\cite{BGT1995}) and Whitney umbrellas (\cite{FH2013}), we have several criteria on singularities of $h_{\bm{v}}$ and $\mathcal{R}^+$-versality of $H$ (Table~\ref{tab:regular} and \ref{tab:WU}). 

\begin{table}[htbp]
\caption{Criteria of the singularity of $h_{\bm{v}}$ on a regular surface $S$ at $p\in S$ and conditions for $H$ to be a versal unfolding of $h_{\bm{v}}$.}
\label{tab:regular}
\begin{tabular}{clll}
\hline
 & Criteria of the sing. of $h_{\bm{v}}$ & Conditions for $H$ to be versal\\
\hline
$A_1$ & \begin{minipage}{7cm}
$\bm{v}$ is normal to $S$ at $p$ and $p$ is not parabolic
\end{minipage} & Always \\\hline
$A_2$ & \begin{minipage}{7cm}
$\bm{v}$ is normal to $S$ at $p$ and $p$ is a parabolic point, which is not ridge
\end{minipage} & Always\\\hline
$A_3$ & \begin{minipage}{7cm}
$\bm{v}$ is normal to $S$ at $p$ and $p$ is a parabolic point, which is 1st order ridge
\end{minipage} & The parabolic locus is a smooth curve\\
\hline
\end{tabular}
\end{table}
\begin{table}[htbp]
\caption{Criteria of the singularity of $h_{\bm{v}}$ on a Whitney umbrella $S$ at its singular point $p$ and conditions for $H$ to be a versal unfolding of $h_{\bm{v}}$.}
\label{tab:WU}
\begin{tabular}{clll}
\hline
 & Criteria of the sing. of $h_{\bm{v}}$ & Conditions for $H$ to be versal\\
\hline
$A_1$ & \begin{minipage}{7cm}
$\bm{v}$ is normal to $S$ at $p$ and $\bm{v}$ attains no parabolic point over $p$
\end{minipage} & Always \\\hline
$A_2$ & \begin{minipage}{7cm}
A normal $\bm{v}$ attains a parabolic point over $p$, which is not ridge
\end{minipage} & Always\\\hline
$A_3$ & \begin{minipage}{7cm}
A normal $\bm{v}$ attains a parabolic point over $p$, which is 1st order ridge
\end{minipage} & Whitney umbrella is elliptic\\
\hline
\end{tabular}
\end{table}

We expect these observations can be generalized to singular surfaces with more degenerate singularities.
This paper is a trial in this direction. 
Actually, in \cite{Mond1985}, D.~Mond classified $\mathcal{A}$-simple map-germs $(\R^2,0)\to(\R^3,0)$ (Table \ref{tab:A-simple}), and, in this paper, we generalize these observations above for surfaces with $\mathcal{A}$-simple singularities of type $S_k$, $B_k$, $C_k$ and $F_4$. 
Our result (Theorem~\ref{thm:Height}) is summarized  as Table~\ref{tab:sing}. 
The proof is based on differential geometric treatment of singular surfaces via resolutions of singularities. 

\begin{table}[htbp]
\caption{Criteria of the singularity of $h_{\bm{v}}$ on a singular surfaces $S$ with $S_k$, $B_k$, $C_k$ and $F_4$ singularities at its singular point $p$ and conditions for $H$ and $\wt{H}$ to be a versal unfolding of $h_{\bm{v}}$ and $h_{\bm{v},t}$, respectively. }
\label{tab:sing}
\begin{tabular}{clll}
\hline
& Criteria of the sing. of $h_{\bm{v}}$ and $h_{\bm{v},t}$& Conditions for $H$ and $\wt{H}$ to be versal\\
\hline
$A_1$ & \begin{minipage}{7cm}
$\bm{v}$ is normal to $S$ at $p$ and $\bm{v}$ attains no parabolic point over $p$
\end{minipage} & Always \\\hline
$A_2$ & \begin{minipage}{7cm}
A normal $\bm{v}$ attains a parabolic point over $p$, which is not ridge
\end{minipage} & Always\\\hline
$A_3$ & \begin{minipage}{7cm}
A normal $\bm{v}$ attains a parabolic point over $p$, which is 1st order ridge
\end{minipage} & The sing. point is not inflection\\
\hline
\end{tabular}
\end{table}

We show that a unified treatment is possible for $S_k$, $B_k$, $C_k$ and $F_4$ singularities.  
We do not treat the case with $H_k$ singularities, because this requires another type of resolution. 

As an application, we show criteria of the singularities of dual surfaces of our singular surfaces (Theorem~\ref{thm:dual}). 
We obtain sufficient conditions that the dual surfaces have cuspidal edge or swallowtail as singularities. 
 
%The contact of a surface with a plane is also related to the singularities of height functions. 
%The analysis of the contact of a submanifold with degenerate objects (lines, planes, circles, spheres, etc.) is important for understanding geometric properties of the submaifold.
%From this view point, several researchers have studied the differential geometry of submanifolds in Euclidean space (see, for example, \cite{BGT1999, CMR-F2009, MRR1995, Porteous1971}).
 
%However, we can not use the normal forms given in Table \ref{tab:A-simple} as parameterization of the singular surfaces because local differential geometry of the surfaces may not be preserved by diffeomorphisms in the target.

\begin{table}[ht!]
\caption{Classes of $\mathcal{A}$-simple map-germs.}
\centering
\begin{tabular}{ccc}
\hline
Name & Normal form & $\mathcal{A}$-codim.\\\hline
Immersion & $(x, y, 0)$ & $0$ \\
Whitney umbrella ($S_0$) & $(x, y^2, x y)$ & $2$ \\
$S_k^\pm$ & $(x, y^2, y^3 \pm x^{k+1} y)$,\, $k\geq1$ & $k+2$\\
$B_k^\pm$ & $(x, y^2, x^2 y \pm y^{2k+1})$,\, $k\geq2$ & $k+2$\\
$C_k^\pm$ & $(x, y^2, x y^3 \pm x^k y)$,\, $k\geq3$ & $k+2$\\
$F_4$ & $(x, y^2, x^3 y + y^5)$ & $6$\\
$H_k$ & $(x, x y + y^{3k-1}, y^3)$,\, $k\geq2$ & $k+2$\\\hline
\end{tabular}
  
(When $k$ is even, $S_k^+$ is equivalent to $S_k^-$, and $C_k^+$ to $C_k^-$.)
\label{tab:A-simple}
\end{table}
 
The paper is organized as follows.
In Section 2, we recall several geometric notions for singular surfaces we treat, introduced in \cite{FH2015}. 
We use finite successions of blowing-ups of singular surfaces parameterized by smooth map-germs $\mathcal{A}$-equivalent to one of $S_k$, $B_k$, $C_k$ and $F_4$ singularities.
In Section 3, we describe criteria of singularities of the height functions and versality of the families $H$ and $\wt{H}$ (Theorem~\ref{thm:Height}).
In Section 4, we discuss relations of the singularities of height functions with the parabolic locus of our singular surfaces (Theorem~\ref{thm:parabolic}).
In Section 5, as an application, we show criteria for singularities of dual surfaces of our singular surfaces in terms of the several geometric notions introduced in Section 2 (Theorem~\ref{thm:dual}).

\section{Differential geometry for singular surfaces.}

If $S$ is a regular surface in $\R^3$ and contains the origin, its tangent plane can be given by the $xy$-plane by using a rotation in $\R^3$.
Then $S$ is defined as the graph of the equation $z = f(x,y)$ for some function $f$, and taking $x$ and $y$-axes to be the principal directions at the origin, $f$ can be locally expressed as
\[
f(x,y) = \frac12(k_1 x^2 + k_2 y^2) + O(x,y)^3,
\]
where $k_1$ and $k_2$ are the principal curvature at the origin.

Two map-germs $f, g:(\R^2,0)\to(\R^3,0)$ are said to be \textit{$\mathcal{A}$-equivalent} if $g = \Phi \circ f \circ \phi^{-1}$ for some germs of diffeomorphisms $\phi$ and $\Phi$ of, respectively, the source and target. 
A map-germ $f:(X,x_0)\to(Y,y_0)$, where $X$ and $Y$ aree topological space, is said to be \textit{$\mathcal{A}$-simple} if there is a finite number of equivalence classes such that if $f$ is embedded in any family $F:(X\times P, (x_0,p_0)) \to (Y,y_0)$, then for every $(x,p)$ in a sufficiently small neighbourhood of $(x_0,p_0)$, the germ of $f_p$ at $x$, where we define $f_p(x)=F(x,p)$, lies in one of these equivalence classes.

If $S$ is a singular surface parameterized by a smooth map-germ $g:(\R^2,0)\to(\R^3,0)$ whose 2-jet is $\mathcal{A}$-equivalent to $(u,v^2,0)$, then $g$ can be expressed as the following normal form by using change of coordinates in the source and a rotation in the target which do not change the geometry of $S$.

\begin{proposition}
[\cite{FH2015}]
 \label{prop:normal_form}
Let $g:(\R^2,0)\to(\R^3,0)$ be a map-germ whose $2$-jet is $\mathcal{A}$-equivalent to $(u,v^2,0)$. 
Then, after using rotations in the target and changes of coordinates in the source, we can reduce $g$ to the form \begin{equation}
\label{eq:normal_form_corank1}
(u,p(u,v),q(u,v)), 
\end{equation}
where
\begin{align*}
p(u,v)&=
\frac12 v^2 + \sum_{i=2}^k \frac{b_i}{i!} u^i + O(u,v)^{k+1},\,\\
q(u,v)&=
\frac12 a_{2,0} u^2 + \sum_{m=3}^k\sum_{i+j=m} \frac{a_{i,j}}{i!j!}u^i v^j + O(u,v)^{k+1}.
\end{align*}
\end{proposition}

Assume that $S$ is a singular surface parameterized by a smooth map-germ $g:(\R^2, 0)\to(\R^3, 0)$ of corank 1 at the origin $0$. 
The tangent plane degenerates to a line at the singular point $g(0,0)$.
We call such a line a {\it tangent line}. 
The plane passing through $g(0,0)$ perpendicular to the tangent line is called the {\it normal plane}.

We consider the orthogonal projection of $S$ onto the normal plane. 
The projection can be expressed as $(u,v)\mapsto(p(u,v),q(u,v))$.
Set the group $\mathcal{G} = \mathrm{GL}_2(\R) \times \mathrm{GL}_2(\R)$ which acts on $(j^2 p, j^2 q)$. 
The list of $\mathcal{G}$-orbits is given in Table \ref{tab:singular_point} (see \cite{Gibson1979} for example). 
The singular points of $S$ are classified in terms of the $\mathcal{G}$-class of $(j^2 p, j^2 q)$ in Table \ref{tab:singular_point}. 
From Proposition \ref{prop:normal_form}, if $j^2 g(0,0)$ is $\mathcal{A}$-equivalent to $(u,v^2,0)$ then the singular point $g(0,0)$ is a hyperbolic, inflection or degenerate inflection point.
\begin{table}[ht!]
\centering
\caption{The classification of the singular points.}
\begin{tabular}{cc}
\Bline
$\mathcal{G}$-class & Name \\\hline
$(x^2,y^2)$ & hyperbolic point\\
$(x y,x^2-y^2)$ & elliptic point\\
$(x^2,x y)$ & parabolic point\\
$(x^2\pm y^2,0)$ & inflection point\\
$(x^2,0)$ & degenerate inflection point\\
$(0,0)$ & degenerate inflection point\\\Bline
\end{tabular}
\label{tab:singular_point}
\end{table}

There exists non-zero vector $\eta\in T_{\bm{0}}\R^2$ such that
$dg_{\bm{0}}(\eta)=0$.
We call $\eta$ a {\it null vector} (cf. \cite{KRSUY2005}).
Suppose that $j^2{g}(0,0)$ is $\mathcal{A}$-equivalent to $(u,v^2,0)$.
The plane passing through $g(0,0)$ spanned by $\xi g(0,0)$ and $\eta\xi g(0,0)$ is called the {\it principal plane}, where $\xi \in T_{\bm{0}}\R^2$ is a non-zero vector so that $\{\xi,\eta\}$ is linearly independent and $\zeta f$ is the directional derivative of a vector valued function $f$ along the direction $\zeta$.
The vector, in the normal plane, normal to the principal plane is called the {\it principal normal vector}. 

We remark that the definitions of the tangent line, normal plane, principal plane, principal normal vector and type of singular points do not depend on the choice of coordinates in the source and choice of $\eta$.

A regular plane curve in the parameter space passing through $(0,0)$ is called a {\it tangential curve} if it is transverse to $\eta$ at $(0,0)$. 
Let $\gamma(t)$ be a parameterization of the tangential curve. 
Clearly, $g\circ\gamma$ is tangent to the tangent line of the singular surface. 
We denote $\Gamma$ by a family of tangential curves $\gamma$.
A member $\Gamma_0$ of the family is a {\it characteristic tangential curve} if the curvature of the orthogonal projection of $g\circ\Gamma_0$ onto the principal plane at $g(0,0)$ has an extremum value $\kappa_0$.
Note that tangential curves tangent to the characteristic tangential curve are characteristic tangential curves.

\begin{figure}[ht!]
\begin{minipage}[t]{0.49\textwidth}
\centering
\includegraphics[height=0.75\textwidth,clip]{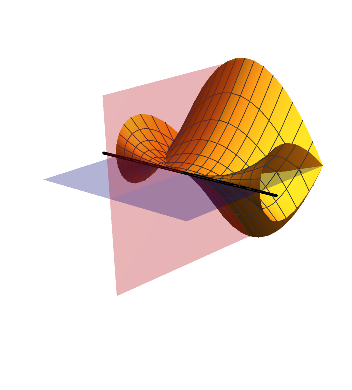}
\end{minipage}\hfill
\begin{minipage}[t]{0.49\textwidth}
\centering
\includegraphics[height=0.75\textwidth,clip]{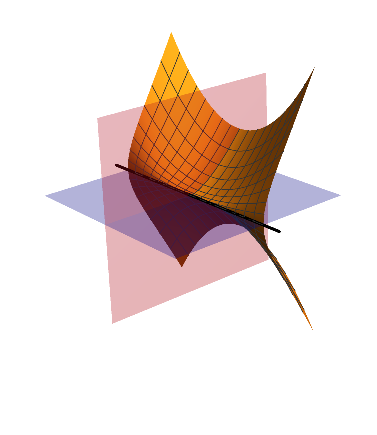}
\end{minipage}
\caption{The tangent line, normal plane and principal plane of $S_1^-$ (left) and $S_1^+$ (right).}
\end{figure}

\begin{remark}
\label{rem:Normal}
Assume that a singular surface is parameterized by $g:(\R^2,0)\to(\R^3,0)$ given in \eqref{eq:normal_form_corank1}.  
We can easily show that the tangent line is the $x$-axis and the normal plane is the $yz$-plane, where $(x,y,z)$ is the usual Cartesian  coordinate system of $\R^3$. 
Furthermore, the null vector can be chosen as $\eta = \partial_v$, and thus the principal plane is the $xy$-plane and $\pm \partial_z$ are the principal normal vectors.

We can take  $\Gamma=(u, c_1 u + c_2 u^2 + O(u^3))$ as the family of tangential curves.
The 2-jet of $g\circ\Gamma$ are given by $(u, (b_2 + c_1^2) u^2/2, a_{2,0} u^2/2)$.
It follows that tangential curves tangent to the $u$-axis are the characteristic tangential curves, and thus the singular point $g(0,0)$ is an inflection (resp. degenerate inflection) point if and only if $a_{2,0} = 0$ (resp. $a_{2,0} = b_2 = 0$).

By using the above argument, it is easily shown that the singular point $g(0,0)$ is an inflection point if and only if $g\circ\gamma$ have at least 3-point contact (inflectional tangent) with the principal plane at $g(0,0)$, and that the inflection point $g(0,0)$ is degenerate if and only if $\kappa_0 = 0$.
\end{remark}

\begin{proposition}[\cite{FH2015}]
\label{prop:CriteriaOfA}
Necessary and sufficient conditions for $g$ given in \eqref{eq:normal_form_corank1} to be $\ma$-equivalent to one of $S_k$, $B_k$, $C_k$, and $F_4$ are as follows: 
\begin{align*}
S_1:&\quad \uwave{a_{2,1}\ne0},\ a_{0,3} \ne0,\\
S_{k\geq2}:&\quad \uwave{a_{2,1} = \cdots = a_{k,1} = 0,\ a_{k+1,1}\ne0},\ a_{0,3}\ne0,\\ 
B_2:&\quad a_{0,3} = 0,\ \uwave{a_{2,1} \ne0},\ 3a_{0,5}\,a_{2,1} - 5a_{1,3}^2 \ne0,\\
B_{k\geq3}:&\quad a_{0,3} = 0,\ \uwave{a_{2,1} \ne0},\ 3a_{0,5}\,a_{2,1} - 5a_{1,3}^2 = 0,\ \xi_3 = \cdots = \xi_{k-1} = 0,\ \xi_k \ne0,\\
C_k:&\quad a_{0,3} = 0,\ \uwave{a_{2,1} = \cdots = a_{k-1,1} = 0,\ a_{k,1} \ne0},\ a_{1,3} \ne0,\\
F_4:&\quad a_{0,3} = 0,\ \uwave{a_{2,1} = 0,\ a_{3,1} \ne0},\ a_{1,3} = 0,\ a_{0,5} \ne0,
\end{align*}
where
\[
\xi_n = \sum_{i=0}^n \sum_{j\geq 1} \dfrac{a_{i,2j-1}\,c_2^{m_2}\cdots c_n^{m_k}}{m_2!\cdots m_k!\,(2j-1)!},\quad \sum_{l=2}^n m_l = i,\quad \sum_{l=2}^n (l-1)m_l = n - j + 1
\]
and $c_2$, $\ldots$, $c_k$ are constants determined by
\[
\sum_{i = 1}^n \sum_{j\geq 1} \dfrac{a_{i,2j-1}\,c_2^{l_2}\,c_3^{l_3}\cdots c_n^{l_n}}{l_2!\,l_3!\cdots l_n!\,(2n-1)!}=0,\quad \sum_{m=2}^n l_m= i -1,\quad \sum_{m=2}^n (m-1)l_m =n-j,\quad n=2,\ldots,k.
\]
\end{proposition}

Let $S$ be a singular surface parameterized by $g$ in \eqref{eq:normal_form_corank1}, and let $g$ be $\mathcal{A}$-equivalent to one of $S_k$, $B_k$, $C_k$ and $F_4$ singularities.
From Proposition \ref{prop:CriteriaOfA}, the condition that
\begin{align}
\label{eq:assumption}
a_{2,1} \ne0 \quad\text{or}\quad a_{2,1} = \cdots = a_{n,1} = 0,\ a_{n+1,1} \ne0
 \quad\text{for some}\quad n\geq2.
\end{align}
holds.
Consider maps
\[
\wt{\Pi}_{n+1}:\R\times S^1\to\R^2, \quad (r,\theta)\mapsto(r\cos\theta,r^{n+1}\cos^n\theta\sin\theta)\quad(n=1\text{ if }a_{21}\ne0),
\]
and
\[
\Pi_{n+1}:\mathcal{M}\to\R^2,\quad [(r,\theta)]\mapsto(r\cos\theta,r^{n+1}\cos^n\theta\sin\theta)\quad(n=1\text{ if }a_{21}\ne0), 
\]
where $\mathcal{M}=\R\times S^1/(r,\theta)\sim(-r,\theta+\pi)$.
The exceptional set $X=\Pi_{n+1}^{-1}(0,0)=\{(r,\theta)\,|\,r \cos\theta = 0\}$.

\begin{proposition}[\cite{FH2015}]
\label{prop:unit_normal}
The unit normal vector $\wt{\bm{n}}=\bm{n}\circ\wt{\Pi}_{n+1}$ to $S$ in the coordinates $(r,\theta)$ is extendible near $X$, and $\wt{\bm{n}}$ can be expressed as
\[
\wt{\bm{n}}(r,\theta) = \left(n_{11}r + O(r^2),\, n_{20} + n_{21}r + n_{22}r^2 + O(r^3),\, n_{30} + n_{31}r + n_{32}r^2 + O(r^3)\right)
\]
where
\[
n_{20} = -\dfrac{a_{n+1,1}\cos\theta}{\ma(\theta)},\quad n_{30} = \dfrac{(n+1)!\sin\theta}{\ma(\theta)}
\]
and the coefficients $(n_{11}, n_{21}, n_{22}, n_{31}, n_{32})$ are trigonometric polynomials with coefficients depending on the $4$-jet and $a_{i,1}$ $(n+1\leq i \leq n+3)$ of $g$, expressed in \upshape{(A.1)} to \upshape{(A.5)} in \cite{FH2015}.
Here,
\[
\ma(\theta) = \sqrt{a_{n+1,1}^2\cos^2\theta+((n+1)!)^2\sin^2\theta}.
\]
\end{proposition}

\begin{proposition}[\cite{FH2015}]
The principal curvatures $\tk{i}=\kappa_i\circ\wt{\Pi}_{n+1}$ $(i=1,2)$  of $S$ in the coordinates $(r,\theta)$ are given by
\begin{align}
\label{eq:blow_k1}
\tk{1}(r,\theta)&=k_{10}+k_{11}r+k_{12}r^2+O(r^3),\\ 
\label{eq:blow_k2}
\tk{2}(r,\theta)&=\frac1{r^{2k+2}}(k_{20}+k_{21}r+O(r^2)),
\end{align}
where
\begin{align}
\label{eq:blow_k10}
k_{10}&= \frac{-a_{n+1,1}\, b_2\cos\theta +(n+1)!\,a_{2,0}\sin\theta}{\ma(\theta)},\\
\label{eq:blow_k20}
k_{20}&=-\frac{((n+1)!)^2a_{n+1, 1}}{\ma(\theta)^3 \cos^{2n - 1}\theta},
\end{align}
and the coefficients $k_{11}$, $k_{12}$, and $k_{21}$ are trigonometric polynomials with coefficients depending on the $4$-jet and $a_{i,1}$ $(n+1\leq i \leq  n+3)$ of $g$, expressed in \upshape{(A.16)} to \upshape{(A.18)} in \cite{FH2015}. 
\end{proposition}

From \eqref{eq:blow_k1} to \eqref{eq:blow_k20}, it follows that the Gaussian curvature $\wt{K}=\tk{1}\tk{2}$ is given by
\begin{equation*}
% \label{eq:gaussian}
\wt{K}(r,\theta) = \frac1{r^{2n+2}}\left(\frac{((n+1)!)^2 a_{n+1,1}(a_{n+1,1}\,b_2\cos\theta - (n+1)!\,a_{2,0}\sin\theta)}{\ma(\theta)^4\cos^{2n - 1}\theta}+O(r)\right).
\end{equation*}
We say that a point $(0,\theta_0)$ is an \textit{elliptic}, \textit{hyperbolic} or \textit{parabolic point over the singularity} of $S$ if $r^{2n+2}\wt{K}(0,\theta_0)$ is positive, negative, or zero, respectively.

A ridge point of a surface in $\R^3$ was first studied in details by Porteous \cite{Porteous1971} as a point where the distance squared function on the surface has an $A_{\geq 3}$-singularity.
It is also a point where one principal curvature has an extremum value along the corresponding line of curvature.
A point where one principal curvature has an extremum value along the other line of curvature is also important.
Such a point is called the sub-parabolic point, which was first studied in details by Bruce and Wilkinson \cite{BW1991} from the viewpoint of  folding maps.
 It is also a point where the lines of curvature have geodesic inflections.

As mentioned above, by using succession of blowing-ups, the unit normal vector is extendible and we can discus asypmtotic behavior near the singularity of our singular surface.
The principal vectors can be also extendible near the singularity.
So we can consider directional derivatives of the principal curvatures along principal vectors near the singularity, and thus ridge and sub-parabolic points can be defined near the singularity.

Let $\tv{i}$ denote the lifted principal vector of the principal vector $\bm{v}_i$ of $S$ by $\wt{\Pi}_{n+1}$.
The directionl derivetives of $\tk{i}$ along $\tv{i}$ are give by the followings:
\begin{align*}
\tv{1}\tk{1}(r,\theta) & = \frac{a_{n+1,1}\Delta_1^{(n+1)}(\theta)\cos\theta}{\ma(\theta)^2}+O(r),\\
\tv{1}^2\tk{1}(r,\theta) & = \dfrac{a_{n+1,1}\bigl(a_{n+1,1}\Delta_2^{(n + 1)}(\theta)\cos\theta - (n + 1)!\, a_{1,2} \Delta_1^{(n+1)}(\theta)\sin\theta\bigr)\cos\theta}{\ma(\theta)^3} + O(r),\\ 
\tv{2}\tk{1}(r,\theta)& = \dfrac1{r^{4 n + 3}}\left(-\dfrac{3 ((n + 1)!)^5 a_{n+1,1} \Delta_3^{(n+1)}(\theta) \sin\theta \cos^{3 - 4n}\theta}{\ma(\theta)^4} + O(r)\right),
\end{align*}
where
\begin{align*}
\Delta_1^{(n+1)}(\theta) & = a_{n+1,1}\,b_3\cos\theta - (n+1)!\,a_{3,0}\sin\theta, \\
\Delta_2^{(n+1)}(\theta) & = - (a_{n+1,1}\,b_4 \cos\theta - (n+1)!\,a_{4,0} \sin\theta)\cos\theta\\
&\quad + 3(a_{2,0}^2 + b_2^2)(a_{n+1,1}\,b_2\cos\theta - (n+1)!\,a_{2,0}\sin\theta)\cos\theta + 12 a_{2,1}\sin^2\theta,\\
\Delta_3^{(n+1)}(\theta) & = a_{n+1,1}\,a_{2,0}\cos\theta_0 - (n+1)!\,b_2\sin\theta_0.
\end{align*}
We define the ridge and sub-parabolic points over the singularity of $S$ are as follows:

\begin{definition}[\cite{FH2015}]
\label{prop:ridge}
Let $\cos\theta_0\ne0$.
\begin{enumerate}
\item 
A point $(0,\theta_0)$ is a \textit{ridge point relative to $\tv{1}$ over the singularity of $S$} if $\Delta_1^{(n+1)}(\theta_0)=0$. 
Moreover, the \textit{ridge point $(0,\theta_0)$ is a first (resp. second or higher) order ridge point relative to $\tv{1}$ over the singularity of $S$} if $\Delta_2^{(n+1)}(\theta_0)\ne0$ (resp. $=0$). 
\item 
A point $(0,\theta_0)$ is a \textit{sub-parabolic point relative to $\tv{2}$ over the singularity of $S$} if $\Delta_3^{(n+1)}(\theta_0)=0$.
\end{enumerate}
\end{definition}

If $g$ is $\mathcal{A}$-equivalent to one of $S_k$, $B_k$, $C_k$ and $F_4$ singularities, we obtain $\wt{\bm{n}}$, $\tk{i\,}$, and $\tv{i\,}$ via $\wt{\Pi}_m$ as shown in Table \ref{tab:BlowUp}. 
Hence we have the following lemma.
\begin{table}[ht!]
\caption{Correspondence between the type of $\ma$-singularity and $\wt{\Pi}_n$.}
\centering
\begin{tabular}{ccccc}
\Bline
$\ma$-type & $S_k$ & $B_k$ & $C_k$ & $F_4$ \\
\hline\\[-12pt]
$\wt{\Pi}_m$ & $\wt{\Pi}_{k+1}$ & $\wt{\Pi}_2$ & $\wt{\Pi}_k$ & $\wt{\Pi}_3$\\ 
\Bline
\end{tabular}
\label{tab:BlowUp}
\end{table}
\begin{lemma}[\cite{FH2015}]
The ridge point relative $\tv{1}$ and sub-parabolic point relative $\tv{2}$ over the singularity of $S$ are determined by $\Delta_{i}^{(m)}$ as shown in Table \ref{tab:Ridge}. 
\end{lemma}
\begin{table}[ht!]
\caption{Criteria for ridge and sub-parabolic points.}
\centering
\begin{tabular}{ccccc}
\Bline
$\ma$-type & $S_k$ & $B_k$ & $C_k$ & $F_4$ \\
\hline\\[-12pt]
$\Delta_i^{(m)}$ & $\Delta_i^{(k+1)}$ & $\Delta_i^{(2)}$ & $\Delta_i^{(k)}$ & $\Delta_i^{(3)}$\\
\Bline
\end{tabular}
\label{tab:Ridge}
\end{table}

\section{Families of height functions on singular surfaces.}

We do not recall here the definition of a versal unfolding.
Please refer, for example, to \cite[Section 8 and 19]{Arnold1986} and \cite[Section 3]{Wall1981}.
 
We consider a families $H$ and $\wt{H}$ of functions on a surface $S$ parameterized by a smooth map-germ $g:(\R^2,0)\to (\R^3,0)$ by
\[
H:(\R^2\times S^2, (0,\bm{v}_0)) \to\R,\quad H(u,v,\bm{v})=\langle g(u,v),\bm{v}\rangle 
\]
and
$$
\wt{H}:(\R^2\times S^2 \times \R,(0,\bm{v}_0,t_0))\to\R, \quad \wt{H}(u,v,\bm{v},t) = H(u,v,\bm{v})-t = \langle f(u,v),\bm{v}\rangle - t. 
$$
We define the function $h_{\bm{v}}(u,v) = H(u,v,\bm{v})$, which is the {\it height function on $S$ along $\bm{v}$}.
We also define the function $\wt{h}_{\bm{v},t}(u,v) = \wt{H}(u,v,\bm{v},t)$, which is the {\it extended height function on $S$ along $\bm{v}$}.
We remark that $H$ (resp. $\wt{H}$) is a $2$-parameter (resp. $3$-parameter) unfolding of $h_{\bm{v}}$ (resp. $\wt{h}_{\bm{v},t}$). 

\begin{theorem}
\label{thm:Height}
Let $S$ be a singular surface parameterized by a smooth map-germ $g:(\R^2,0)\to(\R^3,0)$ which is $\ma$-equivalent to one of $S_k^{\pm}$, $B_k^\pm$, $C_k^\pm$ and $F_4$ singularities, and let $g$ be given in the form \eqref{eq:normal_form_corank1}. 
Suppose that $\bm{v}_0$ is on the normal plane at the singularity, that is, $\bm{v}_0 = \pm\wt{\bm{n}}(0,\theta_0)$, where $\wt{\bm{n}}$ is the well-defined unit normal vector obtained by using $\wt{\Pi}_n$ determined by Table \ref{tab:BlowUp} and $\theta_0 \in (-\pi/2,\pi/2]$. 
\begin{enumerate}
\item 
Suppose that $\bm{v}_0$ is not the \textR{principal} normal vector.
\begin{enumerate}
\item[\upshape(1a)] 
$h_{\bm{v}_0}$ and $\wt{h}_{\bm{v}_0,t_0}$ have an $A_1$ singularity at $(0,0)$ if and only if $(0,\theta_0)$ is not a parabolic point over the singularity of $S$.
When this is the case, $H$ $($resp. $\wt{H})$ is an $\mathcal{R}^+$$($resp. $\mathcal{K})$-versal unfoldings of $h_{\bm{v}_0}$ $($resp. $\wt{h}_{\bm{v}_0,t_0})$, respectively.  
\item[\upshape(1b)]
$h_{\bm{v}_0}$ and $\wt{h}_{\bm{v}_0,t_0}$ have an $A_2$ singularity at $(0,0)$ if and only if $(0,\theta_0)$ is a parabolic point and not a ridge point relative to $\tv{1}$ over the singularity of $S$. 
When this is the case, $H$ $($resp. $\wt{H})$ is an $\mathcal{R}^+$$($resp. $\mathcal{K})$-versal unfolding of $h_{\bm{v}_0}$ $($resp. $\wt{h}_{\bm{v}_0,t_0})$, respectively. 
\item[\upshape(1c)]
$h_{\bm{v}_0}$ and $\wt{h}_{\bm{v}_0,t_0}$ have an $A_3$ singularity at $(0,0)$ if and only if $(0,\theta_0)$ is a parabolic point and first order ridge point relative to $\tv{1}$ over the singularity of $S$. 
When this is the case, $H$ $($resp. $\wt{H})$ is an $\mathcal{R}^+$$($resp. $\mathcal{K})$-versal unfolding of $h_{\bm{v}_0}$ $($resp. $\wt{h}_{\bm{v}_0,t_0})$ if and only if $g(0,0)$ is not an inflection point.  
\item[\upshape(1d)]
$h_{\bm{v}_0}$ and $\wt{h}_{\bm{v}_0,t_0}$ have an $A_{\geq 4}$ singularity at $(0,0)$ if and only if $(0,\theta_0)$ is a parabolic point and second or higher order ridge point relative to $\tv{1}$ over the singularity of $S$.
When this is the case, $H$ $($resp. $\wt{H})$ is not an $\mathcal{R}^+$$($resp. $\mathcal{K})$-versal unfolding of $h_{\bm{v}_0}$ $($resp. $\wt{h}_{\bm{v}_0,t_0})$, respectively. 
\end{enumerate}
\item 
Suppose that $\bm{v}_0$ is the \textR{principal} normal vector.
Then $H$ $($resp. $\wt{H})$ is not an $\mathcal{R}^+$$($resp. $\mathcal{K})$-versal unfolding of $h_{\bm{v}_0}$ $($resp. $\wt{h}_{\bm{v}_0,t_0})$, respectively. 
\begin{enumerate}
\item[\upshape(2a)]
$h_{\bm{v}_0}$ and $\wt{h}_{\bm{v}_0,t_0}$ have an $A_{\geq 2}$ singularity at $(0,0)$ if and only if the singular point $g(0,0)$ of $S$ is not an inflection point. 
\item[\upshape(2b)]
$h_{\bm{v}_0}$ and $\wt{h}_{\bm{v}_0,t_0}$ have a $D_4$ or more degenerate singularity at $(0,0)$ if and only if $g(0,0)$ is an inflection point. 
\end{enumerate}
\end{enumerate}
\end{theorem}

To show Theorem \ref{thm:Height}, we first show criterion for singularities of height functions and versality of these functions in terms of the coefficients in \eqref{eq:normal_form_corank1}. 
We regard $\bm{v} \in S^2$ as the unit vector in $\R^3$ and we write $\bm{v} = (x, y, z)$.

\begin{proposition}
\label{prop:HeightFunction}
Let $g$ be given in the form \eqref{eq:normal_form_corank1}.
Then $h_{\bm{v}_0}$ and $\wt{h}_{\bm{v}_0,t_0}$ is singular at $(0,0)$ if and only if $\bm{v}_0 = (0,y_0,z_0)$.
Moreover, assume that  $h_{\bm{v}_0}$ is singular at $(0,0)$. 
Then 
\begin{enumerate}
\item 
$h_{\bm{v}_0}$ and $\wt{h}_{\bm{v}_0,t_0}$ have an $A_1$ singularity at $(0,0)$ if and only if $y_0 (b_2 y_0 + a_{2,0} z_0)\ne0$. 
When this is the case, $H$ $($resp. $\wt{H})$ is an $\mathcal{R}^+$$($resp. $\mathcal{K})$-versal unfolding of $h_{\bm{v}_0}$ $($resp. $\wt{h}_{\bm{v}_0,t_0})$. 
\item 
$h_{\bm{v}_0}$ and $\wt{h}_{\bm{v}_0,t_0}$ have an $A_2$ singularity at $(0,0)$ if and only if one of the following conditions holds: 
\begin{enumerate}
\item[\textup{(2a)}]
$b_2 y_0 + a_{2,0} z_0 = 0$, $y_0 \ne0$, $b_3 y_0 + a_{3,0} z_0 \ne0$;
\item[\textup{(2b)}]
$\bm{v}_0 = \pm(0,0,1)$, $a_{2,0} \ne 0$, $a_{0,3} \ne 0$.
\end{enumerate}
If condition \textup{(2a)} holds, then $H$ $($resp. $\wt{H})$ is an $\mathcal{R}^+$$($resp. $\mathcal{K})$-versal unfolding of $h_{\bm{v}_0}$ $($resp. $\wt{h}_{\bm{v}_0,t_0})$. 
On the other hand, if condition \textup{(2b)} holds, then $H$ $($resp. $\wt{H})$ is not an $\mathcal{R}^+$ $($resp. $\mathcal{K})$-versal unfolding of $h_{\bm{v}_0}$ $($resp. $\wt{h}_{\bm{v}_0,t_0})$.  
\item 
$h_{\bm{v}_0}$ and $\wt{h}_{\bm{v}_0,t_0}$ have an $A_3$ singularity at $(0,0)$ if and only if one of the following conditions holds: 
\begin{enumerate}
\item[\textup{(3a)}]
$b_2 y_0 + a_{2,0} z_0 = 0$, $y_0 \ne0$, $b_3 y_0 + a_{3,0} z_0 = 0$, $b_4 y_0^2 + a_{4,0} y_0 z_0 - 3a_{2,1}^2 z_0^2\ne0$;
\item[\textup{(3b)}]
$\bm{v}_0 = \pm(0,0,1)$, $a_{2,0} \ne 0$, $a_{0,3} = 0$, $a_{2,0}\, a_{0,4} - 3 a_{1,2}^2 \ne 0$.
\end{enumerate}
If condition \textup{(3a)} holds, then $H$ $($resp. $\wt{H})$ is an $\mathcal{R}^+$$($resp. $\mathcal{K})$-versal unfolding of $h_{\bm{v}_0}$ $($resp. $\wt{h}_{\bm{v}_0,t_0})$ if and only if $a_{2,0} \ne 0$. 
On the other hand, if condition \textup{(3b)} holds, then $H$ $($resp. $\wt{H})$ is not an $\mathcal{R}^+$$($resp. $\mathcal{K})$-versal unfolding of $h_{\bm{v}_0}$ $($resp. $\wt{h}_{\bm{v}_0,t_0})$.  
\item 
$h_{\bm{v}_0}$ and $\wt{h}_{\bm{v}_0,t_0}$ have an $A_{\geq4}$ singularity at $(0,0)$ if and only if one of the following conditions holds: 
\begin{enumerate}
\item[\textup{(4a)}]
$b_2 y_0 + a_{2,0} z_0 = 0$, $y_0 \ne 0$, $b_3 y_0 + a_{3,0} z_0 =0$, $b_4 y_0^2 + a_{4,0} y_0 z_0 - 3 a_{2,1}^2 z_0^2 = 0$;
\item[\textup{(4b)}]
$\bm{v}_0 = \pm(0,0,1)$, $a_{2,0} \ne 0$, $a_{0,3} = 0$, $a_{2,0}\, a_{0,4} - 3 a_{1,2}^2 = 0$.
\end{enumerate}
When this is the case, $H$ $($resp. $\wt{H})$ is not an $\mathcal{R}^+$$($resp. $\mathcal{K})$-versal unfolding of $h_{\bm{v}_0}$ $($resp. $\wt{h}_{\bm{v}_0,t_0})$. 
\item 
$h_{\bm{v}_0}$ and $\wt{h}_{\bm{v}_0,t_0}$ have a singularity of type $D_4$ or more degenerate singularity at $(0,0)$ if and only if $\bm{v}_0 = \pm(0,0,1)$ and $a_{2,0} = 0$.
When this is the case, $H$ $($resp. $\wt{H})$ is not an $\mathcal{R}^+$$($resp. $\mathcal{K})$-versal unfolding of $h_{\bm{v}_0}$ $($resp. $\wt{h}_{\bm{v}_0,t_0})$. 
\end{enumerate}
\end{proposition}

\begin{proof}
Remark that the necessary and sufficient conditions to determine the type of singularities of $h_{\bm{v}_0}$ is same as that of $\wt{h}_{\bm{v}_0,t_0}$.
Moreover, the criteria for $H$ to be an $\mathcal{R}^+$-versal unfolding of $h_{\bm{v}_0}$ is much the same as that for $\wt{H}$ to be a $\mathcal{K}$-versal unfolding of $\wt{h}_{\bm{v}_0,t_0}$. 
So we shall prove the case of $h_{\bm{v}_0}$ and $H$.
 
Since $\partial h_{\bm{v}_0}/\partial u(0,0) = x_0$ and $\partial h_{\bm{v}_0}/\partial v(0,0) = 0$, the function $h_{\bm{v}_0}$ is singular at $(0,0)$ if and only if $x_0 =0$.
 
Assume that $x_0=0$.
Then
\begin{equation}
\label{eq:2JetOfHeightFunction}
j^2 h_{\bm{v}_0}(0,0) = \dfrac12((b_2 y_0 + a_{2,0} z_0) u^2 + y_0 v^2),
\end{equation}
the function $h_{\bm{v}_0}$ has an $A_1$ singularity at $(0,0)$ if and only if $y_0(b_2 y_0 + a_{2,0} z_0) \ne 0$.
Moreover, the singularity of $h_{\bm{v}_0}$ at $(0,0)$ is of type $A_{\geq 2}$  if and only if (i) $b_2 y_0 + a_{2,0} z_0 = 0$, $y_0 \ne 0$, or (ii) $\bm{v}_0 = \pm(0,0,1)$, $a_{2,0} \ne 0$, and it is of type $D_4$ or more degenerate if and only if $\bm{v}_0 = \pm(0,0,1)$, $a_{2,0} = 0$.

We assume that condition (i).
Since $y_0\ne0$, by replacing $v$ by $v - a_{2,1} z_0 u^2/(2 y_0)$,  we can reduce 4-jet of $h_{\bm{v}_0}$ to
\begin{align*}
\begin{split}
j^4h_{\bm{v}_0}(0,0) & = \dfrac12 y_0 v^2 + \dfrac16((b_3 y_0 + a_{3,0} z_0) u^3 + 3 a_{1,2} z_0 u v^2 + a_{0,3}z_0 v^3)\\
&\quad + \dfrac1{24}\left(\dfrac{b_4 y_0^2 + a_{4,0} y_0 z_0 - 3a_{2,1}^2 z_0^2}{y_0} u^4 + 4 c_{3,1} u^3 v + 6 c_{2,2} u^2 v^2 + 4 c_{1,3} u v^3 + c_{0,4}v^4\right),
\end{split}
\end{align*}
where $c_{3,1}, c_{2,2}, c_{1,3}, c_{0,4} \in \R$.
This expression implies that the assertions (2a), (3a) and (4a) hold.

We turn to the case (ii) and assume that condition (ii).
Since $a_{2,0}\ne0$, replacing $u$ by $u - a_{1,2} v^2/(2 a_{2,0})$, we can reduce 4-jet of $h_{{\bm v}_0}$ to
\begin{align*}
j^4h_{\bm{v}_0}(0) & = \pm\left(\dfrac12a_{2,0} u^2 + \dfrac16(a_{3,0} u^3 + 3 a_{2,1} u^2 v + a_{0,3} v^3)\right.\\
& \quad \left. + \dfrac1{24}\left(\hat c_{4,0} u^4 + 4 \hat c_{3,1} u^3 v + 6 \hat c_{2,2} u^2 v^2 + 4 \hat c_{1,3} u v^3 + \dfrac{a_{0,4}\, a_{2,0} - 3 a_{1,2}^2}{a_{2,0}} v^4\right)\right),
\end{align*}
where $\hat c_{4,0}, \hat c_{3,1}, \hat c_{2,2}, \hat c_{1,3} \in \R$. 
This expression implies that the assertions (2b), (3b) and (4b) hold.

We proceed to the proof of the versal unfoldings. 
We skip the proofs of the assertion (1) and (2), since the proofs of (1) and (2) are similar to that of (3).
First, we consider the condition (3a). 
Assume that (3a) holds.
We may assume that $y \ne 0$ near $(u,v,\bm{v}) = (0,0,\bm{v}_0)$.  
We set $y = \pm\sqrt{1-x^2-z^2}$.
Since $A_3$ singularity is $4$-determined, we need to verify the equality
\begin{equation}
\label{eq:VersalityHeight3a}
\mathcal{E}_2 = \left\langle \pd{h_{\bm{v}_0}}{u}, \pd{h_{\bm{v}_0}}{v} \right\rangle_{\mathcal{E}_2} + \left\langle\left.\pd{H}{x}\right|_{\R^2 \times\{\bm{v}_0\}}, \left.\pd{H}{z}\right|_{\R^2 \times\{\bm{v}_0\}}\right\rangle_{\R} + \langle 1 \rangle_{\R} + \langle u, v \rangle_{\mathcal{E}_2}^5
\end{equation}
holds to show that $H$ is an $\mathcal{R}^+$-versal unfolding of $h_{\bm{v}_0}$.
 Replacing $v$ by $v - a_{2,1} z_0 u^2/(2y_0)$, we show that the coefficients of $u^i v^j$ of functions appearing in \eqref{eq:VersalityHeight3a} are given by the following tables:
\[
\mbox{\footnotesize{$
\begin{array}{c|cc|ccc|cccc|c}
& u & v & u^2 & u v & v^2 & u^3 & u^2 v & u v^2 & v^3 & u^4 \\\hline
H_x & \fbox{$1$} & 0 & 0 & 0 & 0 & 0 & 0 & 0 & 0 & 0\\
H_z & 0 & 0 & C_{2,0}/2 & C_{1,1} & C_{0,2}/2 & C_{3,0}/6 & C_{2,1}/2 & C_{1,2}/2 & C_{0,3}/6 & C_{4,0}/24 \\\hline
(h_{\bm{v}_0})_u & 0 & 0 & 0 & 0 & a_{1,2} z_0/2 & \fbox{$c_{4,0}/6$} & c_{3,1}/2 & c_{2,2}/2 & c_{1,3}/6 & c_{5,0}/24 \\
(h_{\bm{v}_0})_v & 0 & \fbox{$y_0$} & 0 & a_{1,2} z_0/2 & a_{0,3} z_0 & c_{3,1}/6 & c_{2,2}/2 & c_{1,3}/2 & c_{0,4}/6 & c_{4,1}/24 \\\hline
u(h_{\bm{v}_0})_u & 0 & 0 & 0 & 0 & 0 & 0 & 0 & a_{1,2}z_0/2 & 0 & \fbox{$c_{4,0}/6$} \\
u(h_{\bm{v}_0})_v & 0 & 0 & 0 & \fbox{$y_0$} & 0 & 0 & a_{1,2}z_0/2 & a_{0,3}z_0 & 0 & c_{3,1}/6 \\
v(h_{\bm{v}_0})_v & 0 & 0 & 0 & 0 & \fbox{$y_0$} & 0 & 0 & a_{1,2}z_0/2 & a_{0,3}z_0 & 0 \\\hline
u^2(h_{\bm{v}_0})_v & 0 & 0 & 0 & 0 & 0 & 0 & \fbox{$y_0$} & 0 & 0 & 0 \\
u v(h_{\bm{v}_0})_v & 0 & 0 & 0 & 0 & 0 & 0 & 0 & \fbox{$y_0$} & 0 & 0 \\
v^2(h_{\bm{v}_0})_v & 0 & 0 & 0 & 0 & 0 & 0 & 0 & 0 & \fbox{$y_0$} & 0 \\\hline
\end{array}$}}
\]
\[
\mbox{\small{$
\begin{array}{c|c|ccccc}
& u^i v^j\ (i + j \geq 3) & u^4 & u^3 v & u^2 v^2 & u v^3 & v^4 \\\hline
u^3 (h_{\bm{v}_0})_v & 0 & 0 & \fbox{$y_0$} & 0 & 0 & 0 \\
u^2 v (h_{\bm{v}_0})_v & 0 & 0 & 0 & \fbox{$y_0$} & 0 & 0 \\
u v^2 (h_{\bm{v}_0})_v & 0 & 0 & 0 & 0 & \fbox{$y_0$} & 0 \\
v^3 (h_{\bm{v}_0})_v & 0 & 0 & 0 & 0 & 0 & \fbox{$y_0$}\\\hline
\end{array}$}}
\]
Here
\[
c_{i,j} = \dfrac{\partial^{i + j} h}{\partial u^i \partial v^j}(0,0) \quad\mbox{and}\quad C_{i,j} = \dfrac{\partial^{i + j + 1} H}{\partial u^i \partial v^j \partial z}(0,0,\bm{v}_0).
\]
We have $c_{4,0} = (b_4 y_0^2 + a_{4,0} y_0 z_0 - 3 a_{2,1}^2 z_0^2)/y_0$ and $C_{2,0} = (a_{2,0} y_0 - b_2 z_0)/y_0$.
Since $y_0 \ne 0$ and $b_4 y_0^2 + a_{4,0} y_0 z_0 - 3 a_{2,1}^2 z_0^2 \ne 0$, the matrix represented by the above tables is of full rank, that is, the equality \eqref{eq:VersalityHeight3a} holds if and only if $a_{2,0} y_0 - b_2 z_0 \ne0$ (i.e., $C_{2,0}\ne0$).
Now we have $b_2 y_0 + a_{2,0} z_0 = 0$ and $y_0 \ne 0$.
It follows that $a_{2,0} y_0 - b_2 z_0 \ne0$ is equivalent to $a_{2,0} \ne 0$.

Next, we consider the condition (3b).
We assume that (3b) holds. 
We may assume that $z \ne 0$ near $(u,v,\bm{v}) = (0,0,\bm{v}_0)$.
We set $z = \pm \sqrt{1 - x^2 - y^2}$.
The unfolding $H$ is an $\mathcal{R}^+$-versal unfolding of $h_{\bm{v}_0}$ if and only if 
\begin{equation}
\label{eq:VersalityHeight3b}
\mathcal{E}_2 = \left\langle\pd{h_{\bm{v}_0}}{u}, \pd{h_{\bm{v}_0}}{v}\right\rangle_{\mathcal{E}_2} + \left\langle\left.\pd{H}{x}\right|_{\R^2 \times\{\bm{v}_0\}}, \left.\pd{H}{y}\right|_{\R^2 \times\{\bm{v}_0\}}\right\rangle_{\R} + \langle 1 \rangle_{\R} + \langle u, v \rangle_{\mathcal{E}_2}^5. 
\end{equation}
Since
\begin{align*}
& \pd{h}{u}(u,v) = \pm \left(a_{2,0}u + \dfrac12(a_{3,0} u^2 + 2 a_{2,1} u v)\right) + O(u,v)^3,\\
& \pd{h}{v}(u,v) = \pm \dfrac12(a_{2,1} u^2 + a_{0,3} v^2) + O(u,v)^3,\\
& \pd{H}{x}(u,v,\bm{v}_0) = u, \\
& \pd{H}{y}(u,v,\bm{v}_0) = \dfrac12(b_2 u^2 + v^2) + O(u,v)^3,
\end{align*}
the equality \eqref{eq:VersalityHeight3b} does not hold. 

Last, we shall prove (4) and (5).
The number of parameters in an $\mathcal{R}^+$-mini-versal unfolding of $A_4$ is 3.
Therefore, $H$ is not an $\mathcal{R}^+$-versal unfolding of $h_{\bm{v}_0}$ having $A_{\geq 4}$ singularity because it is a 2-parameter unfolding. 
For the similar reason, $H$ is not an $\mathcal{R}^+$-versal unfolding of $h_{\bm{v}_0}$ having $D_4$ or more degenerate singularity.
\end{proof}

\begin{proof}[Proof of Theorem \ref{thm:Height}]
First we remark that the condition \eqref{eq:assumption} and the following condition hold.
\[
\bm{v}_0 = (x_0,y_0,z_0) = \pm \left(0,\,-\dfrac{a_{n+1,1}\cos\theta_0}{\mathcal{A}(\theta_0)},\, \dfrac{(n + 1)! \sin\theta_0}{\mathcal{A}(\theta_0)}\right).
\]
 
(1) The proofs of (1a), (1b) and (1d) are similar to that of (1c), so we will omit these proofs and only show the proof of (1c).
Since $y_0 \ne 0$, from Proposition \ref{prop:HeightFunction},  $h_{\bm{v}_0}$ has an $A_3$ singularity at $(0,0)$ if and only if
\begin{align}
\label{eq:HA2}
\begin{split}
& b_2 y_0 + a_{2,0}z_0 = \pm(-a_{n+1,1}\,b_2\cos\theta_0 + (n+1)!\,a_{2,0}\sin\theta_0) = 0
\end{split}\\
\begin{split}
& b_3 y_0 + a_{3,0}z_0 = \pm(-a_{n+1,1}\,b_3\cos\theta_0 + (n+1)!\,a_{3,0}\sin\theta_0) = 0
\end{split}\\
\intertext{and}
\label{eq:HA4}
\begin{split}
& b_4 y_0^2 + a_{4,0} y_0 z_0 - 3a_{2,1}^2 z_0^2\\ 
& = a_{n+1,1}\bigl(-a_{n+1,1}\,b_4 \cos\theta + (n+1)!\,a_{4,0} \sin\theta\bigr)\cos\theta + 12 a_{2,1}^2 \sin^2\theta \ne 0
\end{split}
\end{align}
hold, and $H$ is an $\mathcal{R}^+$-versal unfolding of $h_{\bm{v}_0}$ having an $A_3$ singularity $(0,0)$ if and only if $a_{2,0} \ne 0$. 
From the definitions of the parabolic and first order ridge point over singularity, $(0,\theta_0)$ is the parabolic point and first order ridge point relative to $\tv{1}$ over the singularity of $S$ if and only if \eqref{eq:HA2}--\eqref{eq:HA4} hold. 
Moreover, it follows from Remark \ref{rem:Normal} that $a_{2,0} \ne 0$ if and only if $g(0,0)$ is not the inflection point, and we complete the proof of (1c). 

(2) The statements follow from immediately Proposition \ref{prop:HeightFunction} and the definition of an inflection point. 
\end{proof}

% A direction $\bm{v}$ where $h_{\bm{v}}$ has a
% degenerate singularity  at $(0,0)$ is
% called the {\it degenerate normal direction $($DND$)$}.
% \begin{remark}
%  If $a_{2,0}=0$, then the condition for $h_{\bm{v}}$ to have a
%  degenerate singularity (i.e., non-Morse type singularity) at $(0,0)$ is
%  equivalent to $y_0^2 b_2=0$.
%  It follows that if $a_{2,0} = b_2 = 0$ then for all $\bm{v}$,
%  $h_{\bm{v}}$ has a degenerate singularity at $(0,0)$.
%  Furthermore, if $a_{2,0}=0$ but $b_2 \ne 0$, then there is no direction
%  $\bm{v}$ where $h_{\bm{v}}$ has an $A_{\geq 2}$ singularity at
%  $(0,0)$. 
% \end{remark}

\section{Parabolic sets of singular surfaces}

The locus of points $(u,v)$ where $h_{\bm{v}}$ along some directions $\bm{v}$ has a non-Morse type singularity is the parabolic set.
We call such a direction a \textit{degenerate normal direction}.
Even at a singular point of a singular surface parameterized by a smooth map-germ $g = g(u,v):(\R^2,0)\to(\R^3,0)$, we can define the parabolic set of the surface as the zero set of 
\begin{equation*}
\Sigma(u,v)=(\langle g_{uu},g_u\times g_v\rangle \langle g_{vv},g_u\times g_v\rangle-\langle g_{uv},g_u\times g_v\rangle^2)(u,v).
\end{equation*}
For $g$ given in \eqref{eq:normal_form_corank1}, we have
\begin{align*}
\Sigma(u,v) & = - \frac12 a_{2,0}(a_{2,1} u^2 v - a_{0,3} v^3) + \frac14 a_{2,1}^2 b_2 u^4 - \frac16 (a_{2,0}\,a_{1,3} + 3a_{3,0}\,a_{2,1} - 3a_{2,1}\,a_{1,2}\,b_2) u^3 v\\
&\quad - \frac32 a_{2,1}^2 u^2 v^2 + \frac12 (a_{2,0}\,a_{1,3} + a_{3,0}\,a_{0,3} - 4a_{2,1}\,a_{1,2} - a_{1,2}\,a_{0,3}\,b_2) u v^3\\
&\quad +\frac1{12}(4a_{2,0}\,a_{0,4} + 6a_{2,1}\,a_{0,3} - 12a_{1,2}^2 - 3a_{0,3}^2 b_2) v^4 + O(u,v)^5.
\end{align*}
The parabolic set $\{(u,v)\,|\,\Sigma(u,v) = 0\}$ has a singularity.
The singularities of the parabolic set of singular surfaces parameterized by one of $\mathcal{A}$-simple singularities of $\mathcal{A}_e$-codimension$\leq3$ are investigated in \cite{O-ST2012}.

In \cite{West1995} and \cite{NT2007}, Whitney umbrellas are generically classified into two types in terms of the singularity of the parabolic set in the parameter space. 
A Whitney umbrella whose parabolic set has an $A_1^-$ singularity is classified as \textit{elliptic Whitney umbrella}.
A Whitney umbrella whose parabolic set has an $A_1^+$ singularity is classified as \textit{hyperbolic Whitney umbrella}.
At the transition between two types, there is a Whitney umbrella whose
parabolic set has an $A_2$ singularity. 
Such a Whitney umbrella is classified as \textit{parabolic Whitney umbrella}.

The parabolic set on an elliptic Whitney umbrella is locally formed two intersecting smooth curves.
There are two degenerate normal direction of the elliptic Whitney umbrella at its singular point (see \cite{FH2013,O-ST2012}), and each branch of the parabolic set is associated with one of the two direction.  
Moreover, the torsion and its derivatives of the each branch relate to the type of the degenerate singularity of the hight function. 

\begin{theorem}[\cite{O-ST2012}, Theorem 2.2]
\label{thm:Tari2012}
Let $P_i(t)$ $(i=1,2)$ be parameterizations of branches of the parabolic set on an elliptic Whitney umbrella $P_i(0)$ being Whitney umbrella singularity, and $\tau_i(t)$ denote by the the torsion of $P_i(t)$. 
Then the height function on the elliptic Whitney umbrella along the degenerate normal direction associated to the branch $P_i$ has a singularity at the Whitney umbrella singularity of type 
\begin{align*}
A_2 &\Longleftrightarrow \tau_i(0)\ne0 \\
A_3 &\Longleftrightarrow \tau_i(0)=0,\,\tau_i^\prime(0)\ne0 \\
A_4 & \Longleftrightarrow \tau_i(0)=\tau_i^\prime(0)=0,\ \tau_i''(0) \ne0 \\
\end{align*}
\end{theorem}

We shall consider degenerate normal directions of a singular surface parameterized by a smooth map-germ $g$ whose 2-jet $j^2g(0)$ is $\ma$-equivalent to $(u,\,v^2\,0)$. 
We conclude from \eqref{eq:2JetOfHeightFunction} that if the singular point $g(0)$ is a degenerate inflection point then any direction in the normal plane at $g(0)$ is the degenerate direction, and that if $g(0)$ is non-degenerate inflection point then there is no degenerate normal direction. 
The following theorem is an analog result to that of the Theorem \ref{thm:Tari2012}. 
\begin{theorem}
\label{thm:parabolic}
Let $S$ be a singular surface parameterized by a smooth map-germ $g:(\R^2,0)\to(\R^3,0)$ which is $\mathcal{A}$-equivalent to one of $S_k$, $B_k$, $C_k$ and $F_4$ singularities, and let the singular point $g(0)$ of $S$ be not inflection point.  
\begin{enumerate}
\item 
There is a branch $P$, which is a characteristic tangential curve, of the parabolic set of $S$.
\item 
The branch $P$ has at least $m$-point contact with its tangent line at $(0,0)$, where $m$ is as shown in Table \ref{tab:contact}.
\begin{table}[ht!]
\caption{}
\centering
\begin{tabular}{ccccc}
\Bline
$\ma$-type of singular surfaces & $S_k$ & $B_k$ & $C_k$ & $F_4$ \\ \hline
$m$ & $k+1$ & 2 & $k$ & 3\\
\Bline
\end{tabular}
\label{tab:contact}
\end{table}
\item  
Let $\mathcal{L}(t)$ be a parameterization of $P$ on $S$ with $\mathcal{L}(0)$ being the singular point $g(0)$, and let $\bm{b}(t)$ and $\tau(t)$ denote by the unit binormal vector and the torsion of $\mathcal{L}(t)$, respectively.  
Then $h_{\bm{v}}$ on $S$ along $\bm{v}=\pm{\bm{b}}(0)$ has a singularity at $(0,0)$ of type 
\begin{align*}
A_2&\Longleftrightarrow\tau(0)\ne0\\
A_3&\Longleftrightarrow\tau(0)=0,\ \tau^\prime(0)\ne0,\\
A_{\geq4}&\Longleftrightarrow\tau(0)=\tau^\prime(0)=0.
\end{align*}
\end{enumerate}
\end{theorem}
\begin{proof}
We may assume that $g$ is given by \eqref{eq:normal_form_corank1} with \eqref{eq:assumption} and $a_{2,0} \ne 0$.
Note that curves, in the source, which are tangent to the $u$-axis are the characteristic tangential curves (see Remark \ref{rem:Normal}).
Set
\[
A_m = A_m(u,v) = \sum_{i+j=m}\dfrac{a_{i,j}}{i!j!}u^i v^j,\quad A = A(u,v) = \sum_{m = 3}^k A_m,\quad B = B(u) = \sum_{i = 3}\dfrac{b_i}{i!}u^i.
\]
We have 
\begin{align*}
L'=&\langle g_{uu},g_u\times g_v\rangle =-A_vB_{uu}+v(a_{2,0}+A_{uu}),\\
M'=&\langle g_{uv},g_u\times g_v\rangle = vA_{uv},\ \textrm{and}\\
N'=&\langle g_{vv},g_u\times g_v\rangle =-A_v+vA_{vv}.
\end{align*}
We thus have 
\begin{align*}
L'N'-(M')^2 & = (A_vB_{uu}-v(a_{2,0}+A_{uu}))(A_v-vA_{vv})-v^2A_{uv}^2\\
& = A_v^2B_{uu}-vA_v(a_{2,0}+A_{uu}+A_{vv}B_{uu})+v^2((a_{2,0}+A_{uu})A_{vv}-A_{uv}^2),
\end{align*}
and
\begin{align*}
\partial_v(L'N'-(M')^2)|_{v=0} & = 2(A_vA_{vv})B_{uu}-A_v (a_{2,0}+A_{uu}+A_{vv}B_{uu})|_{v=0}\\
& = -A_v(a_{2,0}+A_{uu}-A_{vv}B_{uu})|_{v=0}.
\end{align*}
So we obtain that 
\begin{align*} 
L'N'-(M')^2|_{u=0} = & \dfrac{a_{2,0}\, a_{0,3}}2 v^3 + O(v^4),\\
L'N'-(M')^2|_{v=0} = & \frac{a_{n+1,1}^2}{((n+1)!)^2}u^{2(n+1)}B_{u u} + O(u^{2n+3}),\\
\partial_v(L'N'-(M')^2)|_{v=0} = & -\frac{a_{2,0}\,a_{n+1,1}}{(n+1)!}u^{n+1} + O(u^{n+2}).
\end{align*}
Since $B_{uu} = b_2 + b_3 u + O(u^2)$, so the Newton polygon of $L'N'-(M')^2$ looks like as in Figure \ref{fig:Newton}. 
Hence, the locus $L' N' -(M')^2 = 0$ has a local branch $P$ defined by  
$$
v = \frac{a_{n+1,1}\,b_2}{(n+1)!\, a_{2,0}}u^{n+1} + O(u^{n+2}), 
$$
which is the characteristic tangential curve and has at least $(n+1)$-point contact with its tangent line at $(0,0)$.
The number $n$ is determined by the type of $\mathcal{A}$-singularity, and Proposition \ref{prop:CriteriaOfA} gives the table of the assertion (2).

We now turn to the proof of (3).
First, we assume that $a_{2,1}\ne0$.
Then the branch $P$ has at least 2-point contact with its tangent line at $(0,0)$ and can be parameterized, in the parameter space, by $t\mapsto(t,\ a_{2,1}\,b_2 t^2/(2a_{2,0}) + O(t^3))$.
Hence,
\begin{equation}
\label{eq:branch}
j^4\mathcal{L}(0) = \left(t,\, \dfrac{b_2}2 t^2 + \dfrac{b_3}6 t^3 + \dfrac{a_{2,0}^2\, b_4 + a_{2,1}^2\, b_2^2}{24 a_{2,0}^2} t^4,\, \dfrac{a_{2,0}}2 t^2 + \dfrac{a_{3,0}}6 t^3 + \dfrac{a_{4,0}\,a_{2,0} + 6 a_{2,1}^2\, b_2}{24} t^4 \right).
\end{equation}
Straightforward calculations show that
\begin{align}
\label{eq:binormal}
\bm{b}(0) & = \left(0,\, -\dfrac{a_{2,0}}{\sqrt{a_{2,0}^2 + b_2^2}},\, \dfrac{b_2}{\sqrt{a_{2,0}^2 + b_2^2}}\right), \\
\label{eq:tau}
\tau(0) & = \dfrac{a_{3,0}\, b_2 - a_{2,0}\, b_3}{a_{2,0}^2 + b_2^2}, \\
\label{eq:tau'}
\tau'(0) & = -\dfrac{2(a_{3,0}\, b_2 - a_{2,0}\, b_3)(a_{2,0}\, a_{3,0} + b_2\, b_3)}{(a_{2,0}^2 + b_2^2)^2} + \dfrac{a_{4,0}\, a_{2,0}\, b_2 + 3 a_{2,1}^2\, b_2^2 - a_{2,0}^2\, b_4}{a_{2,0}(a_{2,0}^2 + b_2^2)}. 
\end{align}
We set $\bm{v} = \pm \bm{b}(0)$. 
The height function on $S$ in $\bm{v}$ is expressed as $h_{\bm{v}} = \mp a_{2,0} v^2/\sqrt{a_{2,0}^2 + b_2^2} + O(u,v)^3$ and has an $A_{\geq 2}$ singularity at $(0,0)$. 
Replacing $v$ by $v + a_{2,1}\, b_2 u^2/(2 a_{2,0})$, we show that the coefficients of $u^3$, $u^2 v$ and $u^4$ of $h_{\bm{v}}$ are, respectively,
\begin{equation}
\label{eq:height_in_binormal}
\dfrac{\pm (a_{3,0} b_2 - a_{2,0} b_3)}{6\sqrt{a_{2,0}^2 + b_2^2}},\quad 0,\quad \dfrac{\pm (a_{4,0} a_{2,0} b_2 + 3 a_{2,1}^2 b_2^2 - a_{2,0}^2 b_4)}{24 a_{2,0}\sqrt{a_{2,0}^2 + b_2^2}}.
\end{equation}
Therefore, the assertion follows form \eqref{eq:tau}--\eqref{eq:height_in_binormal}.

Next, we assume that $a_{2,1} = \cdots = a_{n,1} = 0$ and $a_{n+1,1} \ne 0$ for some $n\geq 2$.
Then the branch $P$ has at leat $(n+1)$-point contact with its tangent line at $(0,0)$ and can be parameterized by $t\mapsto (t,\ c t^{n+1} + O(t^{n+2}))$ $(c\in\R)$.
Hence, $j^4\mathcal{L}(0)$ is given by \eqref{eq:branch} with $a_{2,1}$ replaced by $0$, namely,  
\begin{equation*}
j^4\mathcal{L}(0) = \left(t,\ \dfrac{b_2}2 t^2 + \dfrac{b_3}6 t^3 + \dfrac{a_{2,0}^2\, b_4}{24 a_{2,0}^2} t^4,\ \dfrac{a_{2,0}}2 t^2 + \dfrac{a_{3,0}}6 t^3 + \dfrac{a_{4,0}\,a_{2,0}}{24} t^4 \right).
\end{equation*}
Therefore, the assertion follows from same argument above with $a_{2,1}$ replaced by $0$.

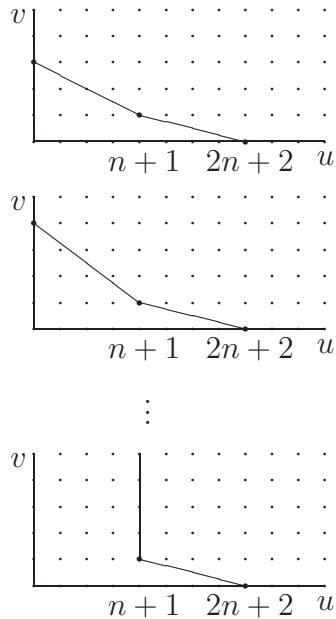
\begin{figure}[ht]
\centering
\begin{picture}(110,70)(-10,-20)
\put(0,0){\line(0,1){50}}
\put(0,0){\line(1,0){110}}
\put(0,30){\line(2,-1){40}}
\put(40,10){\line(4,-1){40}}
\put(80,0){\circle*{2}}
\put(40,10){\circle*{2}}
\put(0,30){\circle*{2}}
\put(28,-11){$n + 1$}
\put(65,-11){$2 n + 2$}
\put(107,-9){$u$}
\put(-9,45){$v$}
\multiput(0,0)(0,10){6}{
\multiput(0,0)(10,0){12}{\circle*{1}}}
\end{picture}
   
\begin{picture}(110,70)(-10,-20)
\put(0,0){\line(0,1){50}}
\put(0,0){\line(1,0){110}}
\put(0,40){\line(4,-3){40}}
\put(40,10){\line(4,-1){40}}
\put(80,0){\circle*{2}}
\put(40,10){\circle*{2}}
\put(0,40){\circle*{2}}
\put(28,-11){$n + 1$}
\put(65,-11){$2 n + 2$}
\put(107,-9){$u$}
\put(-9,45){$v$}
\multiput(0,0)(0,10){6}{
\multiput(0,0)(10,0){12}{\circle*{1}}}
\end{picture}

$\vdots$
\vspace*{10pt}
   
\begin{picture}(110,70)(-10,-20)
\put(0,0){\line(0,1){50}}
\put(0,0){\line(1,0){110}}
\put(40,10){\line(0,1){40}}
\put(40,10){\line(4,-1){40}}
\put(80,0){\circle*{2}}
\put(40,10){\circle*{2}}
\put(28,-11){$n + 1$}
\put(65,-11){$2 n + 2$}
\put(107,-9){$u$}
\put(-9,45){$v$}
\multiput(0,0)(0,10){6}{
\multiput(0,0)(10,0){12}{\circle*{1}}}
\end{picture}
\caption{Newton polygons of $L'N'-(M')^2$.}
\label{fig:Newton}
\end{figure}
\end{proof}

To compare the singularities of the height functions on between regular surfaces and singular surfaces from the view point of the torsion of (the branch of) the parabolic set,  we see the following proposition. 
\begin{proposition}
\label{prop:SinguOfHeightFunctionOnRegular}
Let $S$ be a regular surface parameterized by a smooth map $g:(\R^2,0)\to(\R^3,0)$, and let $g(0,0)$ is a parabolic point but not an umbilic point.
Suppose that the parabolic set in the parameter space is not singular at $(0,0)$, and that $\mathcal{L}(t)$ is the parameterization of the parabolic set on $S$ with $\mathcal{L}(0) = (0,0,0)$.
Let denote the unit binormal vector, the curvature and the torsion of $\mathcal{L}(t)$ by $\bm{b}(t)$, $\kappa(t)$ and $\tau(t)$, respectively. 
Let $h_{\bm{v}}$ be the height function on $S$ in the normal direction $\bm{v} = \pm \bm{n}(0,0)$.
If $\kappa(0) \ne0$, then the followings hold.  
\begin{enumerate}
\item 
The function $h_{\bm{v}}$ has an $A_2$ singularity at $(0,0)$ if and only if $\bm{b}(0)\ne\pm\bm{n}(0,0)$.
\item 
The function $h_{\bm{v}}$ has an $A_{\geq 3}$ singularity at $(0,0)$ if and only if $\bm{b}(0) = \pm\bm{n}(0,0)$.
If $h_{\bm{v}}$ has an $A_{\geq3}$ singularity at $(0,0)$, then $\tau(0) = 0$.
\item 
Assume that $h_{\bm{v}}$ has an $A_{\geq 3}$ singularity at $(0,0)$. 
Then if $\tau'(0) \ne 0$ then the singularity is of type $A_3$.
Moreover, if the singularity is of type $A_{\geq 4}$ then $\tau'(0) = 0$. 
\end{enumerate}
\end{proposition}
\begin{proof}
We may assume that $g$ is given by Monge form
\[
g(u,v) = \left(u, v, f(u,v) \right), \quad f(u,v) = \dfrac12k_2 v^2 + \sum_{i + j = 3}^k \dfrac1{i! j!}a_{i,j}u^i v^j + O(u,v)^{k+1} \quad (k_2 \ne 0).
\]
Then we have $\bm{n}(0,0) = (0,0,1)$, and $h_{\bm{v}} = \pm k_2 v^2/2 + O(u,v)^3$. 
Replacing $v$ by $v - a_{2,1} u^2/(2 k_2)$, we show that the coefficients of $u^3$, $u^2 v$ and $v^4$ of $h_{\bm{v}}$ are, respectively,
\[
\pm\dfrac{a_{3,0}}{6},\quad 0,\quad \pm\dfrac{a_{4,0}\,k_2 - 3 a_{2,1}^2}{24 k_2}.
\]
It turns out that $h_{\bm{v}}$ has a singularity of type $A_{\geq3}$ and $A_{\geq 4}$ if and only if, respectively, 
\[
a_{3,0} = 0,\quad\mbox{and}\quad a_{3,0} = a_{4,0}\, k_2 - 3 a_{2,1}^2 = 0.
\]

The parabolic set is given by the zero set of $g_{uu}g_{vv} - g_{u,v}^2$. 
Now we write $\gamma(t) = (u(t), v(t))$ for the parameterization of the parabolic set with $(u(0), v(0)) = (0,0)$.
Since the parabolic set is not singular at $(0,0)$, we have $(u'(0), v'(0))\ne(0,0)$. 
The parabolic set on $S$ is given by $\mathcal{L}(t) = g \circ \gamma(t)$.
Straightforward calculations show that
\[
\mathcal{L}'(0) \times \mathcal{L}''(0) = (k_2 v'(0)^3,\ -k_2 u'(0)v'(0),\ u'(0) v''(0) - v'(0)u''(0)).
\]
It follows that if $\kappa(0) \ne 0$ then $(v'(0),v''(0))\ne(0,0)$. 
Assume that $\kappa(0)\ne0$.
Then $\bm{b}(0) = \pm \bm{n}(0,0) = \pm(0,0,1)$ if and only if $v'(0) = 0$. 
Since
\[
\pd{}{u}(g_{uu}g_{vv}-g_{u,v}^2)(0,0) = a_{3,0}\quad\mbox{and}\quad \pd{}{v}(g_{uu}g_{vv}-g_{u,v}^2)(0,0) = a_{2,1},
\]
$v'(0) = 0$ if and only if $a_{3,0} = 0$.
 
Assume that $a_{3,0} = 0$.
Since the parabolic set is not singular at $(0,0)$, we have $a_{2,1} \ne 0$. 
Then by the implicit function theorem we show that $\gamma(t)$ can be expressed as
\begin{equation*}
\gamma(t) = \left(t,\ \dfrac{2a_{2,1}^2 - a_{4,0}\, k_2}{2 a_{2,1} k_2} t^2 + \dfrac{a_{4,0}\, a_{3,1}\, k_2^2 - 2 a_{4,0}\, a_{2,1}\, a_{1,2}\, k_2 + a_{2,1}^3\, a_{1,2}}{2 a_{2,1}^2\, k_2^2} t^3 + O(t^4)\right)
\end{equation*}
near $(0,0)$. 
We remark that $2 a_{2,1}^2 - a_{4,0}\, k_2 \ne 0$ because $\kappa(0)\ne0$.
Straightforward calculations show
\begin{equation*}
\tau(0) = 0, \quad \tau'(0) = \dfrac{(8a_{2,1}^2 - 3 a_{4,0}\, k_2)(3 a_{2,1}^2 - a_{4,0}\, k_2)}{a_{2,1}(2 a_{2,1}^2 - a_{4,0}\, k_2)},  
\end{equation*}
which completes the proof. 
\end{proof}

\section{Singularities of dual surfaces}

Let a smooth map $f:U\to \R^3$ be a parameterization of a regular surface $M$, where $U \subset \R^2$ is a open subset.
We consider the family $\wt{H}$ of the extended height functions on $M$
\[
\wt{H}:U\times S^2\times \R\to\R, \quad \wt{H}(u,v,\bm{v},t) = H(u,v,\bm{v})-t = \langle f(u,v),\bm{v}\rangle - t. 
\]
Since $\wt{H}_u(u,v,\bm{v},t) = \wt{H}_v(u,v,\bm{v},t) = 0$ if and only if $\bm{v} = \pm \bm{n}(u,v)$. 
Hence, the discriminant set $\mathcal{D}(\wt{H})$ of $\wt{H}$ is given by
\begin{equation*}
\mathcal{D}(\wt{H}) = \{(\pm \bm{n}(u,v),\pm\langle f(u,v),\bm{v}\rangle)\,|\,(u,v)\in U\}. 
\end{equation*}
Set a smooth map
\[
\Psi:S^2 \times (\R\setminus \{0\}) \to (\R^3\setminus \{0\}), \quad \Psi(\bm{v},t) = t \bm{v}.
\]
We show that $\Psi(\mathcal{D}(\wt{H})) = \langle f(u,v),\bm{n}(u,v)\rangle \bm{n}(u,v)$ under the assumption that $\langle f(u,v),\bm{n}(u,v)\rangle \ne 0$.
If necessary, we have the condition $\langle f(u,v),\bm{n}(u,v)\rangle \ne 0$ by using isometries in $\R^3$, which do not change the geometry of $M$.
Therefore, we may assume $\langle f(u,v),\bm{n}(u,v)\rangle \ne 0$.
A {\it dual surface} of $M$ is a surface parameterized by 
\[
f^*:U \to \R^3, \quad f^*(u,v) = \langle f(u,v),\bm{n}(u,v)\rangle \bm{n}(u,v). 
\]
We remark that $\Psi(\mathcal{D}(\wt{H})) = f^*(U)$.
 
So, we define the dual of a singular surface $S$ parameterized by a smooth map $g$ of corank $1$ as follows.
Let $\bm{v}\in S^2$ be in the normal plane at the singularity of $S$.
A dual surface of $S$ is $\Psi(\mathcal{D}(\wt{H}))$ if $\langle g(u,v), \bm{n}(u,v)\rangle \ne 0$ at regular points $(u,v)$ and $\langle g(u_0,v_0), \bm{v} \rangle \ne 0$ at the singualrity $(u_0,v_0)$.
 
By the definition of a dual surface of $S$, the singularity of the dual surface of $S$ coincides with that of $\mathcal{D}(\wt{H})$.
It is well-known that the singularity of the discriminant set of the $\mathcal{K}$-versal unfolding of a function having $A_2$ singularity is a cuspidal edge.
It is also well-known that the singularity of the discriminant set of the $\mathcal{K}$-versal unfolding of a function having $A_3$ singularity is a swallowtail.
Here, a singularity is called a \textit{cuspidal edge} or \textit{swallowtail} if the corresponding map-germs is $\mathcal{A}$-equivalent to
\[
f_c:=(u^2, u^3, v)\quad \mbox{or}\quad f_s:=(3 u^4 + u^2 v, 4 u^3 + 2 u v, v),
\]
respectively. 
It follows that if $\wt{H}$ is a $\mathcal{K}$-versal unfolding of $\wt{h}_{\bm{v},t}(u,v) = \wt{H}(u,v,\bm{v},t)$ having $A_2$ (resp. $A_3$) singularity then the dual surface has a singularity of type cuspidal edge (resp. swallowtail).  
%\begin{figure}[ht]
%\begin{center}
%\includegraphics[width=0.6\textwidth]{light_CE_SW01.eps}
%\end{center}
%\caption{Cuspidal edge (left); swallowtail (right).}
%\label{fig:cuspidal_edge_and_swallowtail}
%\end{figure}

When a singular surface $S$ is parameterized by a smooth map $g$ in \eqref{eq:normal_form_corank1}, we consider a map 
\[
\bar{g}(u,v) = g(u,v) + \bm{p},
\]
where $\bm{p}$ satisfies the condition $\langle \bm{p}, \wt{\bm{n}}(0,\theta) \rangle \ne 0$.
Since translations preserve the geometry of a surface, we regard the dual of $\bar{S}$ parameterized by $\bar{g}$ as the dual of $S$. 
 
\begin{theorem}
\label{thm:dual}
Let $S$ be a singular surface parameterized by a smooth map-germ $g:(\R^2,0)\to(\R^3,0)$ which is $\mathcal{A}$-equivalent to one of $S_k$, $B_k$, $C_k$ and $F_4$ singularities, and let $g$ be given in \eqref{eq:normal_form_corank1}.
Assume that the singular point $g(0)$ of $S$ be not inflection point. 
\begin{enumerate}
\item 
If $(0,\theta_0)$ is a parabolic point over the singularity, the singularity of the dual surface $\wt{g}^*$ of $S$ at $\wt{g}^*(0,\theta_0)$ is a cuspidal edge. 
\item 
If $(0,\theta_0)$ is a parabolic point and a first order ridge relative to $\wtilde{\bm{v}}_1$ but not a sub-parabolic point relative to $\wtilde{\bm{v}}_2$ over the singularity, then the singularity of the $\wt{g}^*$ at $\wt{g}^*(0,\theta_0)$ is a swallowtail.
\end{enumerate}
\end{theorem}
 
\begin{proof}
We consider the family of extended height functions
\[
\wt{H}:(\R^2\times S^2 \times \R, (0,\bm{v}_0,t_0)) \to \R, \quad \wt{H}(u,v,\bm{v},t) = \langle \bar{g}(u,v), \bm{v}\rangle - t.
\]
It is clear that the condition for $\wt{h}_{\bm{v},t}$ on $\bar{S}$ to have a versal unfolding coincide with that on $S$ to have a versal unfolding.
Hence, by using Theorem \ref{thm:Height} we complete the proof.
\end{proof}
\begin{remark}
Similar results are obtained for dual surfaces of regular surfaces (\cite{BGT1995}) and for dual surfaces of singular surfaces with cuspidal edge (\cite{Teramoto2015}) by analyzing the singularly of the height function on these surfaces. 
\end{remark}
%%%%%%%%%%%%%%%%%%%%%%%%%%%%%%%%%%%%%%%%%%%%%%%%%%

\end{document}